\documentclass[a4paper,12pt,reqno]{amsart}

\usepackage[utf8]{inputenc}
\usepackage{amsmath, amsthm, amsfonts, amssymb}
\usepackage[foot]{amsaddr}
\usepackage{stmaryrd}	
\usepackage{graphicx}
\usepackage{subcaption}
\usepackage{todonotes}
\usepackage{xcolor}
\usepackage{tikz}
\usepackage{todonotes}
\usepackage{hyperref} 
\usepackage[mathlines, pagewise]{lineno} 
\usepackage{pb-diagram}
\usepackage{float}
\usepackage{enumitem}
\setlist[itemize]{noitemsep, topsep=0pt}
\setlist[enumerate]{noitemsep, topsep=0pt}

\newcommand{\N}{\mathbb{N}}
\newcommand{\Z}{\mathbb{Z}}
\newcommand{\R}{\mathbb{R}}
\newcommand{\C}{\mathbb{C}}
\newcommand{\II}{\mathbb{I}}

\newcommand{\E}{\mathbb{E}}
\renewcommand{\P}{\mathbb{P}}

\newcommand{\be}{\begin{equation}}
\newcommand{\ee}{\end{equation}}

\newcommand{\supp}{\ensuremath{\mathrm{supp}}}

\definecolor{dark-green}{RGB}{0,200,0}

\newcommand{\david}[1]{{\color{blue}{#1}}}

\newtheorem{theorem}{Theorem}
\newtheorem{lemma}[theorem]{Lemma}
\newtheorem{corollary}[theorem]{Corollary}
\newtheorem{prop}[theorem]{Proposition}

\theoremstyle{definition}
\newtheorem{remark}[theorem]{Remark}

\newtheorem{algorithm}{Algorithm}

\DeclareMathOperator{\median}{med}

\title[A simple universal algorithm for integration]{A simple universal algorithm for high-dimensional integration}

\author{Takashi Goda$^1$ and David Krieg$^2$}
\address{$^1$Graduate School of Engineering, The University of Tokyo, Japan}
\address{$^2$Department of Mathematics, University of Passau, Germany.}
\email{goda@frcer.t.u-tokyo.ac.jp, david.krieg@uni-passau.de}
\thanks{The work of T.G. is supported by JSPS KAKENHI Grant Number 23K03210}

\date{\today}

\allowdisplaybreaks


\begin{document}

\begin{abstract}
We present a simple universal algorithm
for high-dimensional integration
which has the optimal error rate (independent of the dimension)
in all weighted Korobov classes
both in the randomized and the deterministic setting.
Our theoretical findings are complemented by numerical tests.
\end{abstract}

\maketitle

Lattice rules are a powerful tool for estimating the integral of a high-dimensional function $f\colon [0,1]^d \to \C$. We refer to the monograph \cite{DKP22}.
Given a positive integer $p$ (typically a prime) and a generating vector $z\in \{1:p-1\}^d$, where $\{1:p-1\}$ denotes the set $\{1,2,\ldots,p-1\}$, the lattice rule $Q_p^z$ is defined by
\[
 Q_p^z(f) \,:=\, \frac{1}{p} \sum_{k=0}^{p-1} f\left( \left\{ \frac{k z}{p} \right\}  \right).
\]
Here, $\{\cdot\}$ denotes the (entry-wise) fractional part of a vector.

It is a very active research topic how to choose a good generating vector $z$, see \cite{DG21,DKLP15,EKNO21,Kuo03,NC06,SR02} among many others
or \cite[Chapter~3]{DKP22} for an overview.
Usually, the choice of $z$ highly depends on the precise model assumptions for the integrand $f$.
On the other hand, practical applications usually do not come with precise model assumptions.
In this paper, we want to propagate a very simple 
method to deal with this problem
and propose a randomized quadrature rule 
which is both easy to implement and universally near-optimal
for a wide range of function classes.
 
Namely, we choose $p$ as a random prime between $n/2$ and $n$ 
and $z$ randomly from the full set of generating vectors,
repeat this for several times, 
and take the median of the acquired estimates of the integral.
Precisely, we propose the following randomized algorithm.
Here, $\mathcal{P}_n$ denotes the set of primes between $\lceil n/2 \rceil+1$ and $n$,
and $h\colon \N \to [1,\infty)$ is an arbitrary function
with $h(n)\to \infty$ for $n\to \infty$. 
We think of a function that increases very slowly, 
for instance, we may take $h(n) = \max(1,\log n)$ or even $h(n) = \max(1,\log \log n)$.

\begin{samepage}
\begin{algorithm}\label{alg:median}
Let $n,d\in\N$ and 
let $N=2\lceil h(n) \log_2(n)\rceil+1$. 
\begin{enumerate}
\item For $1\le k \le N$, choose a random prime $p_k$ uniformly from $\mathcal{P}_n$
and a random vector $z_k$ uniformly from $\{1:p_k-1\}^d$.
\item For $f\colon [0,1]^d \to \C$, put
\[
M_n(f) \,:=\, M_n^{(p_k,z_k)_{k}}(f) \,:=\, \median\Big\{ Q_{p_k}^{z_k}(f) \ : \ 1\le k \le N \Big\}.
\]
\end{enumerate}
\end{algorithm}
\end{samepage}

Here, the median of complex numbers
is defined by taking the median of their real and imaginary parts separately.
The algorithm $M_n$  
uses no more than 
$c\,n\log n \cdot h(n)$ function values,
where $c$ is an absolute constant.

Algorithm~\ref{alg:median} combines several tricks, none of which are new:
\begin{itemize}
    \item For univariate functions, already Bakhvalov observed in \cite{Bah61} 
    that choosing the number of points randomly between $n/2$ and $n$
    can improve the expected error by a factor $n^{-1/2}$.
    This was extended to the multivariate setting for lattice rules with generating vectors of the form $z = (1, g, \ldots, g^{d-1})$ for $g \in \{1:p-1\}$ in the same work \cite{Bah61}. With regard to tractable error bounds in high dimensions, 
    the idea of using ``random $p$'' has been revisited in \cite{KKNU19}, and constructive algorithms have been proposed in subsequent works \cite{DGS22, G24, KNW23}.
    \item 
    Although it can be difficult to find a good generating vector explicitly, 
    it is known that the set of good generating vectors is relatively large, see \cite[Theorem~2]{DSWW06} and \cite[Section~2.3]{KKNU19}. 
    Thus, if we take the generating vector at random from the full set of generating vectors,
    there is at least a decent chance that we obtain a good quadrature result.
    \item Whenever we have a method that results in good quadrature with a constant probability larger than $1/2$, we can apply this method several times and take the median in order to get good quadrature with a very high probability. This is known as the \emph{median trick} and was exploited, for instance, in \cite{CXZ24,GL22,GSM24,HR23,KNR19,KR19,PO23,PO24}. 
\end{itemize}

The contribution of our work is the combination of all of these tricks
in order to obtain a simple universal algorithm with near-optimal error bounds in a wide range of function classes.
Our main result can be stated as follows.
Here, $e^{\rm det}_{d,\alpha,\gamma}$ and $e^{\rm ran}_{d,\alpha,\gamma}$ denote the worst-case error on the Korobov class with smoothness $\alpha>1/2$ and product weights $\gamma \in \R_+^d$ on $[0,1]^d$, in the deterministic and the randomized setting, respectively.
All these quantities will be defined in detail in Section~\ref{sec:preliminaries}.
Note that, although Algorithm~\ref{alg:median} is a randomized algorithm, the estimate $M_n^{(p_k,z_k)_{k}}(f)$ is deterministic if $(p_k,z_k)_{k=1,\ldots,N}$ are fixed. This allows us to consider the worst-case error for each realization $M_n^{(p_k,z_k)_{k}}$.

\begin{theorem}\label{thm:main}
Algorithm~\ref{alg:median} has the following properties.
For any $\alpha>1/2$ and $\gamma\in (0,1]^\N$,
and any $\varepsilon >0$,
there are positive constants $n_0=n_0(\alpha,\varepsilon,\gamma,d)$ and $C=C(\alpha,\varepsilon,\gamma,d)$ such that the following holds.
\begin{enumerate}
\item For any 
$n\ge n_0$,
we have
\begin{equation*}
e^{\rm ran}_{d,\alpha,\gamma}(M_n^{(p_k,z_k)_k}) 
\,\le\, C \, n^{-\alpha-1/2+\varepsilon}.
\end{equation*}
\item For any $n\ge 2$,
we have
\begin{equation*}
e^{\rm det}_{d,\alpha,\gamma}(M_n^{(p_k,z_k)_k}) \,\le\, C \, n^{-\alpha+\varepsilon},
\end{equation*}
with probability at least $1-n^{-h(n)}$.
\end{enumerate}
Here, both $n_0$ and $C$ are independent of the dimension $d$
if
$\gamma \in \ell_{1/\alpha}$. 
\end{theorem}

We refer to Theorem~\ref{thm:randomized-detailed} and Theorem~\ref{thm:deterministic-detailed}
for further details. There, we obtain bounds on the respective errors not only for product weights but also for general weights. A simple argument in Remark~\ref{rem:product_weight} shows the dimension-independence of $n_0$ and $C$.

\begin{remark}[Non-periodic functions]\label{rem:nonperiodic}
    The Korobov classes with smoothness $\alpha>1/2$ contain only smooth functions that are one-periodic in each variable. 
    By applying the tent transformation $\phi: [0,1]\to [0,1]$, defined as $\phi(x)=1-|2x-1|$, to each node of the lattice rules component-wise, the results of this paper can be extended to the weighted half-period cosine spaces. These spaces contain smooth functions that are not necessarily periodic \cite{DNP14}. We also refer to \cite[Section~4]{DGS22} as a relevant work.
    Hence, for non-periodic integrands $f\colon [0,1]^d\to\C$, we propose the algorithm
    \[
     M_n^*(f) \,:=\, M_n(f\circ \Phi),
    \]
    where $\Phi$ is applied component-wise.
    We do not present the theoretical analysis for this algorithm,
    but it is not hard to see that the algorithm is reasonable since $f\circ \Phi$ is one-periodic and has the same integral as $f$.
    It is also possible to replace the tent transformation with an infinitely smooth transformation in order to preserve higher smoothness of the integrand, see, e.g., \cite{NUU17} and \cite[Section~5]{Ull17}, as well as the references therein. 
    One caution, however, is that such methods are likely to become intractable in high dimensions, see \cite{KSW07}.
\end{remark}

\begin{remark}
    The universality of quadrature algorithms (whether deterministic or randomized) with respect to different degrees of smoothness has been explored through various approaches; see \cite{GDMS24, KN17, Pan24, Ull17} among others. Randomized algorithms that are universal with respect to both smoothness and weights have been proposed in \cite{GL22, GSM24}, though only under the worst-case error criterion. The results of this paper exhibit a decisive difference in this regard.
\end{remark}

\section{Preliminaries}
\label{sec:preliminaries}

Let $f$ have an absolutely convergent Fourier series:
\[ f(x)=\sum_{h\in \Z^d}\hat{f}(h)\exp(2\pi i h\cdot x),\]
where $\hat{f}(h)$ denotes the $h$-th Fourier coefficient and $\cdot$ is the usual dot product. Let $\alpha>1/2$ be a real number and $\gamma=(\gamma_u)_{u\subset \N}$ be a sequence of positive real numbers with $0<\gamma_u\leq 1$. The weighted Korobov class with smoothness $\alpha$ and weights $\gamma$, denoted by $\mathcal{H}_{d,\alpha,\gamma}$, is defined as a reproducing kernel Hilbert space with the reproducing kernel
\[ K_{d,\alpha,\gamma}(x,y)=\sum_{h\in \Z^d}\frac{\exp(2\pi i h\cdot (x-y))}{(r_{d,\alpha,\gamma}(h))^2},\]
and the inner product
\[ \langle f,g\rangle_{d,\alpha,\gamma}=\sum_{h\in \Z^d}(r_{d,\alpha,\gamma}(h))^2\hat{f}(h)\overline{\hat{g}(h)}. \]
Here the function $r_{d,\alpha,\gamma}: \Z^d\to \R_{+}$ is defined by
\[ r_{d,\alpha,\gamma}(h)=\gamma_{\supp(h)}^{-1}\prod_{j\in \supp(h)}|h_j|^{\alpha},\]
with $\supp(h):=\{j: h_j\neq 0\}$, where the empty product is set to $1$.
The induced norm of this class is denoted by $\| f\|_{d,\alpha,\gamma}=\sqrt{\langle f,f\rangle_{d,\alpha,\gamma}}$.

It is well-known that when $\alpha$ is an integer, the Korobov class norm involves the mixed partial derivatives of periodic functions $f$ up to order $\alpha$ in each variable, see \cite[Section~2.4]{DKP22}.

For the sum of the Fourier weights,
we use the notation
\begin{align}\label{eq:def-V}
 V_d&(\alpha,\gamma)
 := \sum_{h\in\Z^d\setminus\{0\}} \frac{1}{r_{d,\alpha,\gamma}(h)}
 = \sum_{\emptyset \neq u\subseteq \{1:d\}}\,\sum_{h\in (\Z\setminus\{0\})^u} \gamma_u\, \prod_{j\in u}|h_j|^{-\alpha} \notag \\ 
 &= \sum_{\emptyset \neq u\subseteq \{1:d\}}\,\gamma_u\, \Big(\sum_{k\in \Z\setminus\{0\}}  |k|^{-\alpha}\Big)^{|u|}
 =\sum_{\emptyset \neq u\subseteq \{1:d\}}\gamma_u\cdot (2\zeta(\alpha))^{|u|},
\end{align}
where $\zeta$ denotes the Riemann zeta function.

Given $p$ and $z\in \{1:p-1\}^d$, the worst-case error of the lattice rule $Q_p^z$ in the weighted Korobov class $\mathcal{H}_{d,\alpha,\gamma}$ is defined by
\[ e^{\rm det}_{d,\alpha,\gamma}(Q_p^z):=\sup_{\substack{f\in \mathcal{H}_{d,\alpha,\gamma}\\ \| f\|_{d,\alpha,\gamma}\leq 1}}\left| I_d(f) - Q_p^z(f) \right|, \]
where $I_d(f)$ denotes the true integral of $f$, i.e.,
\[ I_d(f)=\int_{[0,1]^d} f(x)\, \mathrm{d}x.\]
It follows from \cite[Theorem~2.19]{DKP22} that
\[ \left(e^{\rm det}_{d,\alpha,\gamma}(Q_p^z)\right)^2 = \sum_{\substack{h\in \Z^d\setminus \{0\}\\ h\cdot z\equiv_p 0 }}\frac{1}{(r_{d,\alpha,\gamma}(h))^2},\]
where $h\cdot z\equiv_p 0$ means that $h\cdot z$ is congruent to zero modulo $p$.

Consider a randomized algorithm $M_n$ that is defined as a random variable $(M_n^{\omega})_{\omega\in \Omega}$.
Here, $(\Omega,\Sigma,\mathbb P)$ is a probability space and $M_n^{\omega}: \mathcal{H}_{d,\alpha,\gamma} \to \C$ is a deterministic algorithm (e.g., a quadrature rule) for every $\omega\in\Omega$. Assuming that $(M_n^{\omega}(f))_{\omega\in \Omega}$ is measureable for each $f$, the randomized error of $M_n$ in the weighted Korobov class $\mathcal{H}_{d,\alpha,\gamma}$ is then defined by
\[ e^{\rm ran}_{d,\alpha,\gamma}(M_n):=\sup_{\substack{f\in \mathcal{H}_{d,\alpha,\gamma}\\ \| f\|_{d,\alpha,\gamma}\leq 1}}\E_{\omega}\left| I_d(f) - M_n^{\omega}(f) \right|. \]
In particular, the randomized error of our Algorithm~\ref{alg:median} is given by
\begin{align*}
    & e^{\rm ran}_{d,\alpha,\gamma}((M_n^{(p_k,z_k)_k})) \\
    & \quad = \sup_{\substack{f\in \mathcal{H}_{d,\alpha,\gamma}\\ \| f\|_{d,\alpha,\gamma}\leq 1}}\frac{1}{|\mathcal{P}_n|}\sum_{p_1\in \mathcal{P}_n}\frac{1}{(p_1-1)^d}\sum_{z_1\in \{1:p_1-1\}^d}\cdots\\
    & \quad \quad \quad \quad \quad \frac{1}{|\mathcal{P}_n|}\sum_{p_N\in \mathcal{P}_n}\frac{1}{(p_N-1)^d}\sum_{z_N\in \{1:p_N-1\}^d} \left| I_d(f) - M_n^{(p_k,z_k)_{k}}(f) \right|.
\end{align*}

The following lemma shows that there exists a set of good generating vectors $z\in \{1:p-1\}^d$ of relative size $\tau\in (0,1)$ such that the worst-case error is small. Although the result is essentially due to \cite[Theorem~2]{DSWW06}, we give a proof for completeness.

\begin{lemma}\label{lem:set_good_vectors}
    Let $p$ be a prime and $\tau\in (0,1)$. For any $\alpha>1/2$ and $\gamma=(\gamma_u)_{u\subset \N}$ with $0<\gamma_u\leq 1$ and any $d\in\N$, there exists a subset $Z_{p,\tau} =Z_{p,\tau}(\alpha,\gamma,d)$ of $\{1:p-1\}^d$ with $|Z_{p,\tau}|\geq \lceil \tau (p-1)^d\rceil$ such that, for any $z\in Z_{p,\tau}$, 
    \[ e^{\rm det}_{d,\alpha,\gamma}(Q_p^z)\leq \inf_{1/2\leq \lambda<\alpha}\left(\frac{2}{(1-\tau)(p-1)}
    V_d\left(\alpha/\lambda,\gamma^{1/\lambda}\right)\right)^{\lambda} \]
    with $V_d$ as defined in \eqref{eq:def-V}.
\end{lemma}

\begin{proof}
    From the subadditivity of the map $x\mapsto x^c$ for any $x\geq 0$ with a constant $c\in (0,1]$, it holds for any $\lambda\in [1/2,\alpha)$ that
    \begin{align*}
        \left[e^{\rm det}_{d,\alpha,\gamma}(Q_p^z)\right]^{1/\lambda} & \leq \sum_{\substack{h\in \Z^d\setminus \{0\}\\ h\cdot z\equiv_p 0 }}\frac{1}{(r_{d,\alpha,\gamma}(h))^{1/\lambda}}\\
        & =\sum_{\emptyset \neq u\subseteq \{1:d\}}\gamma_u^{1/\lambda}\sum_{\substack{h_u\in (\Z\setminus \{0\})^{|u|}\\ h_u\cdot z_u\equiv_p 0}}\prod_{j\in u}\frac{1}{|h_j|^{\alpha/\lambda}},
    \end{align*}
    where we denote $z_u=(z_k)_{j\in u}$ for a vector $z$. Then we have
    \begin{align*}
        & \frac{1}{(p-1)^d}\sum_{z\in \{1:p-1\}^d}\left[e^{\rm det}_{d,\alpha,\gamma}(Q_p^z)\right]^{1/\lambda} \\
        & \leq \sum_{\emptyset \neq u\subseteq \{1:d\}}\gamma_u^{1/\lambda}\sum_{h_u\in (\Z\setminus \{0\})^{|u|}}\left(\prod_{j\in u}\frac{1}{|h_j|^{\alpha/\lambda}}\right) \frac{1}{(p-1)^d}\sum_{z\in \{1:p-1\}^d}\II_{h_u\cdot z_u\equiv_p 0},
    \end{align*}
    where $\II_{h_u\cdot z_u\equiv_p 0}$ denotes the indicator function that returns $1$ if $h_u\cdot z_u\equiv_p 0$ holds and $0$ otherwise.

    For the equally-weighted average of $\II_{h_u\cdot z_u\equiv_p 0}$ with respect to $z\in \{1:p-1\}^d$, we know that
    \[ \frac{1}{(p-1)^d}\sum_{z\in \{1:p-1\}^d}\II_{h_u\cdot z_u\equiv_p 0} \leq \begin{cases} 1 & \mbox{if $p\mid h_u$,}\\ 1/(p-1) & \mbox{if $p\nmid h_u$,}\\   \end{cases} \]
    where $p\mid h_u$ means that all components of $h_u$ are multiples of $p$ and $p\nmid h_u$ means that there exist at least one component of $h_u$ that is not a multiple of $p$, see, e.g., \cite[Lemma~4]{KKNU19}. Plugging this inequality into the bound above, we obtain
    \begin{align*}
        & \frac{1}{(p-1)^d}\sum_{z\in \{1:p-1\}^d}\left[e^{\rm det}_{d,\alpha,\gamma}(Q_p^z)\right]^{1/\lambda} \\
        & \leq \sum_{\emptyset \neq u\subseteq \{1:d\}}\gamma_u^{1/\lambda}\sum_{\substack{h_u\in (\Z\setminus \{0\})^{|u|}\\ p\mid h_u}}\left(\prod_{j\in u}\frac{1}{|h_j|^{\alpha/\lambda}}\right) \\
        & \quad + \frac{1}{p-1}\sum_{\emptyset \neq u\subseteq \{1:d\}}\gamma_u^{1/\lambda}\sum_{\substack{h_u\in (\Z\setminus \{0\})^{|u|}\\ p\nmid h_u}}\left(\prod_{j\in u}\frac{1}{|h_j|^{\alpha/\lambda}}\right)\\
        & \leq \sum_{\emptyset \neq u\subseteq \{1:d\}}\gamma_u^{1/\lambda}\sum_{h_u\in (\Z\setminus \{0\})^{|u|}}\left(\prod_{j\in u}\frac{1}{|ph_j|^{\alpha/\lambda}}\right) \\
        & \quad + \frac{1}{p-1}\sum_{\emptyset \neq u\subseteq \{1:d\}}\gamma_u^{1/\lambda}\sum_{h_u\in (\Z\setminus \{0\})^{|u|}}\left(\prod_{j\in u}\frac{1}{|h_j|^{\alpha/\lambda}}\right)\\
        & = \sum_{\emptyset \neq u\subseteq \{1:d\}}\frac{\gamma_u^{1/\lambda}}{p^{|u|\alpha/\lambda}}(2\zeta(\alpha/\lambda))^{|u|}+\frac{1}{p-1}\sum_{\emptyset \neq u\subseteq \{1:d\}}\gamma_u^{1/\lambda}(2\zeta(\alpha/\lambda))^{|u|}\\
        & \leq \frac{2}{p-1}\sum_{\emptyset \neq u\subseteq \{1:d\}}\gamma_u^{1/\lambda}(2\zeta(\alpha/\lambda))^{|u|}.
    \end{align*}

    As we now get a bound on the average of $\left[e^{\rm det}_{d,\alpha,\gamma}(Q_p^z)\right]^{1/\lambda}$ over $z\in \{1:p-1\}^d$, Markov inequality states that, for any $\tau\in (0,1)$, there exists a set $Z_{p,\tau,\lambda}\subseteq \{1:p-1\}^d$ with $|Z_{p,\tau,\lambda}|\geq \lceil \tau (p-1)^d\rceil$ such that
    \[ \left[e^{\rm det}_{d,\alpha,\gamma}(Q_p^z)\right]^{1/\lambda}\leq \frac{2}{(1-\tau)(p-1)}\sum_{\emptyset \neq u\subseteq \{1:d\}}\gamma_u^{1/\lambda}(2\zeta(\alpha/\lambda))^{|u|}\]
    or equivalently
    \[ e^{\rm det}_{d,\alpha,\gamma}(Q_p^z)\leq \left(\frac{2}{(1-\tau)(p-1)}\sum_{\emptyset \neq u\subseteq \{1:d\}}\gamma_u^{1/\lambda}(2\zeta(\alpha/\lambda))^{|u|}\right)^{\lambda} \]
    holds for any $z\in Z_{p,\tau,\lambda}$. 
    The right hand side of the previous inequality is a continuous function of $\lambda\in [1/2,\alpha)$ which tends to infinity for $\lambda \to \alpha$. It therefore attains its minimum at some point $\lambda^*\in [1/2,\alpha)$.
    We put $Z_{p,\tau} := Z_{p,\tau,\lambda^*}$. 
    Then
    \[ e^{\rm det}_{d,\alpha,\gamma}(Q_p^z)\leq \inf_{1/2\leq \lambda<\alpha}\left(\frac{2}{(1-\tau)(p-1)}\sum_{\emptyset \neq u\subseteq \{1:d\}}\gamma_u^{1/\lambda}(2\zeta(\alpha/\lambda))^{|u|}\right)^{\lambda} \]
    holds for any $z\in Z_{p,\tau}$. 
    This completes the proof.
\end{proof}

\begin{remark}\label{rem:product_weight}
    Consider the case of product weights, i.e., every $\gamma_u$ for a non-empty subset $u\subseteq \{1:d\}$ is given by $\gamma_u=\prod_{j\in u}\gamma_j$ for $\gamma=(\gamma_1,\gamma_2,\ldots)\in (0,1]^{\N}$. Using $1+x\leq e^x$ for any $x\in \R$, we have
    \begin{align}\label{eq:rem1}
        V_d\left(\alpha/\lambda,\gamma^{1/\lambda}\right) & = -1+\prod_{j=1}^{d}\left( 1+2\gamma^{1/\lambda}_j \zeta(\alpha/\lambda)\right) \notag \\
        & \leq \prod_{j=1}^{d}\exp\left(2\gamma^{1/\lambda}_j \zeta(\alpha/\lambda)\right)=\exp\left(2 \zeta(\alpha/\lambda)\sum_{j=1}^{d}\gamma^{1/\lambda}_j\right) \notag \\
        & \leq \exp\left(2 \zeta(\alpha/\lambda)\left(\sum_{j=1}^{d}\gamma^{1/\alpha}_j\right)^{\alpha/\lambda}\right) =: M_d(\alpha,\gamma,\lambda)
    \end{align}
    Thus, if $\gamma \in \ell_{1/\alpha}$, then $V_d\left(\alpha/\lambda,\gamma^{1/\lambda}\right)$ is bounded independently of the dimension $d$ for any $\lambda\in [1/2,\alpha)$.
\end{remark}

Another important ingredient in our analysis is the following phenomenon of probability amplification, which is known as the median trick. We refer to \cite[Proposition~2.2]{KNR19} and the references therein.

\begin{lemma} \label{lem:median}
    Let $N$ be an odd natural number,
    $X_1,\hdots,X_N$ real random variables,
    and $I \subset \R$ an interval.
    Assume that $\P(X_i \not\in I) \le \tau$ for some $\tau<1/2$. Then
    \[
     \P(\median\{X_1,\hdots,X_N\} \not\in I)
     \,\le\, \frac12 (4\tau (1-\tau))^{N/2}.
    \]
\end{lemma}


\section{Analysis of the randomized error}

Our findings in the randomized setting are based on the following result, which is due to \cite[Theorem~9]{KKNU19}.
Note that \cite[Theorem~9]{KKNU19} focuses on the case of product weights, and only states the result 
for a set of generating vectors of relative size $\tau=1/2$,
but it is easily adapted to general weights and parameters $\tau \in (0,1)$. 
We will need it for $\tau>1/2$ in order to apply the median trick.

\begin{prop}\label{prop:KKNU}
    For any $\alpha>1/2$, any $\gamma=(\gamma_u)_{u\subset \N}$ with $0<\gamma_u\leq 1$, any $\tau\in (0,1)$, and any prime $p\in \N$, let $Z_{p,\tau}\subseteq \{1:p-1\}^d$ with $|Z_{p,\tau}|\geq \lceil \tau (p-1)^d\rceil$ be as given in Lemma~\ref{lem:set_good_vectors}.
    If $p$ is chosen uniformly from $\mathcal{P}_n$ and
    $z$ is chosen uniformly from $Z_{p,\tau}$, then
    \[
        \E \left| I_d(f) - Q_p^z(f) \right|  \,\le\, C_{\lambda,\delta,\tau} \cdot V_d\left(\alpha/\lambda,\gamma^{1/\lambda}\right)^\lambda \cdot n^{-\lambda-1/2+\delta} \cdot \Vert f \Vert_{d,\alpha,\gamma}
    \]
    for all $f\in \mathcal{H}_{d,\alpha,\gamma}$, any $\lambda \in (1/2,\alpha)$, any $\delta\in (0,\min\{\lambda-1/2,1\})$,
    and all $n\ge 4V_d\left(\alpha/\lambda,\gamma^{1/\lambda}\right)/(1-\tau)$.
    Here, $C_{\lambda,\delta,\tau}$ is a positive constant depending only on $\lambda, \delta$ and $\tau$, and 
    $V_d$ is as defined in \eqref{eq:def-V}.
    \end{prop}

    \begin{proof}
        The proof follows the lines of \cite[Theorem~9]{KKNU19}.
        We only need to replace $\mathcal Z_p$ with $Z_{p,\tau}$ and $B_n$ with
        \[
         B_n' \,:=\, \left( \frac{(1-\tau) \, n}{4 V_d\left(\alpha/\lambda,\gamma^{1/\lambda}\right)} \right)^\lambda.
        \]
        Then, since
        \[
         e^{\rm det}_{d,\alpha,\gamma}(Q_p^z) \,\ge\, \max_{\substack{h\in \Z^d\setminus \{0\}\\ h\cdot z\equiv_p 0 }}\frac{1}{r_{d,\alpha,\gamma}(h)},
        \]
        any $z\in Z_{p,\tau}$ satisfies
        \[
            \rho_{\alpha,\gamma}(p,z)
            := \min_{\substack{h\in \Z^d\setminus \{0\}\\ h\cdot z\equiv_p 0 }} r_{\alpha,\gamma}(h)
            \ge \frac{1}{e^{\rm det}_{d,\alpha,\gamma}(Q_p^z)}
            \ge \left(\frac{(1-\tau)(p-1)}{2 V_d\left(\alpha/\lambda,\gamma^{1/\lambda}\right)}\right)^{\lambda}
            \ge B_n',
        \]
        so that we can indeed argue as in \cite{KKNU19}.
        Switching from product weights to general weights does not make any difference in the proof, as long as we assume that $\gamma_u$ is bounded above by 1 for all $u\subseteq \{1:d\}$, which ensures $r_{d,\alpha,\gamma}(h)\geq 1$ for all $h\in \Z^d$.
    \end{proof}

    Proposition~\ref{prop:KKNU} already provides an algorithm with a near-optimal randomized error on the Korobov class $\mathcal{H}_{d,\alpha,\gamma}$.
    Our problem is that the algorithm depends on the parameters $\gamma$ and $\alpha$
    via the set $Z_{p,\tau}$
    and hence the algorithm is not universal.
    We fix this issue by choosing the generating vector uniformly from the full set $\{1:p-1\}^d$.
    Since the relative size of $Z_{p,\tau}$ is at least $\tau$, 
    the resulting algorithm is still good with a probability greater than $1/2$,
    as stated in the following corollary.

\begin{corollary}~\label{cor:constant-probability}
    Let $p$ be uniformly distributed on $\mathcal{P}_n$ and let
    $z$ be uniformly distributed on $\{1:p-1\}^d$.
    Then, for any $\alpha>1/2$, 
    any $\gamma=(\gamma_u)_{u\subset \N}$ with $0<\gamma_u\leq 1$, 
    any $f\in \mathcal{H}_{d,\alpha,\gamma}$, 
    any $\lambda \in (1/2,\alpha)$, any $\delta\in (0,\min\{\lambda-1/2,1\})$,
    and all $n\ge 64V_d\left(\alpha/\lambda,\gamma^{1/\lambda}\right)$,
    it holds with probability at least $7/8$ that
    \begin{equation}\label{eq:cor}
        \left| I_d(f) - Q_p^z(f) \right|  \,\le\, 16\, \tilde{C}_{\lambda,\delta} \cdot V_d\left(\alpha/\lambda,\gamma^{1/\lambda}\right)^\lambda \cdot n^{-\lambda-1/2+\delta} \cdot \Vert f \Vert_{d,\alpha,\gamma},
    \end{equation}
    where $\tilde{C}_{\lambda,\delta}$ is a positive constant depending only on $\lambda$ and $\delta$, and
    $V_d$ is as defined in \eqref{eq:def-V}.
\end{corollary}

\begin{proof}
    We use Proposition~\ref{prop:KKNU} with $\tau=15/16$.
    Thus, if $p$ is uniformly distributed on $\mathcal{P}_n$ and $z$ is uniformly distributed on $Z_{p,\tau}$, we have
    \[
        \E \left| I_d(f) - Q_p^z(f) \right|  \,\le\, \tilde{C}_{\lambda,\delta} \cdot V_d\left(\alpha/\lambda,\gamma^{1/\lambda}\right)^\lambda \cdot n^{-\lambda-1/2+\delta} \cdot \Vert f \Vert_{d,\alpha,\gamma} =: E
    \]
    with a positive constant $\tilde{C}_{\lambda,\delta}:=C_{\lambda,\delta,\tau=15/16}$ depending only on $\lambda$ and $\delta$, where $C_{\lambda,\delta,\tau}$ is the one appearing in Proposition~\ref{prop:KKNU}.
    Markov's inequality implies
    \[
     \P \left( \left| I_d(f) - Q_p^z(f) \right| > 16 E \,\right) \,\le\, \frac{1}{16}.
    \]
    
    Now let $p$ be uniformly distributed on $\mathcal{P}_n$ 
    and $z$ uniformly distributed on the full set $\{1:p-1\}^d$. Then
    \begin{align*}
    \P&\left( \left| I_d(f) - Q_p^{z}(f) \right|  \,>\, 16 E\,\right)\\
    &\le\, \P( z \not\in Z_{p,\tau}) \,+\, \P\left(z \in Z_{p,\tau} \text{ and } \left| I_d(f) - Q_p^{z}(f) \right|  \,>\, 16 E\,\right)\\
    &\le\, \frac{1}{16} + \frac{1}{16} \,=\, \frac18.
    \end{align*}
    This completes the proof. 
\end{proof}


The success probability $7/8$ in Corollary~\ref{cor:constant-probability}
enables us to use the median trick.
Recall that Algorithm~\ref{alg:median} takes the median of several independent copies
of the algorithm from Corollary~\ref{cor:constant-probability}.
We get the following more explicit and general (in terms of weights) version of part~(i) in Theorem~\ref{thm:main}.

\begin{theorem}\label{thm:randomized-detailed}
    For any $\alpha>1/2$, 
    any $\gamma=(\gamma_u)_{u\subset \N}$ with $0<\gamma_u\leq 1$, 
    any $f\in \mathcal{H}_{d,\alpha,\gamma}$, 
    any $\lambda \in (1/2,\alpha)$, any $\delta\in (0,\min\{\lambda-1/2,1\})$,
    and all $n\ge 64 V_d\left(\alpha/\lambda,\gamma^{1/\lambda}\right)$ with $h(n) \ge \lambda+1/2$,
    Algorithm~\ref{alg:median} satisfies
    \[
        \E\, | I_d(f) - M^{(p_k,z_k)_k}_n(f) |
        \,\le\, C_{\lambda,\delta} \cdot V_d\left(\alpha/\lambda,\gamma^{1/\lambda}\right)^{\lambda} \cdot n^{-\lambda-1/2+\delta} \cdot \Vert f \Vert_{d,\alpha,\gamma},
    \]
    where $C_{\lambda,\delta}$ is a positive constant depending only on $\lambda$ and $\delta$, 
    and 
    $V_d$ is as defined in \eqref{eq:def-V}.
\end{theorem}

\begin{proof}
The random variable $M^{(p_k,z_k)_k}_n(f)$ is the median of 
$N$ independent copies of the random variable $Q_p^{z}(f)$ from Corollary~\ref{cor:constant-probability}.
Let $R$ denote the right hand side of \eqref{eq:cor}.
Then we have
\[
 \P\left( \left| \Re(I_d(f)) - \Re(Q_p^z(f)) \right| > R \right) \,\le\, \frac18.
\]
By Lemma~\ref{lem:median}, we get
\[
 \P\left( \left| \Re(I_d(f)) - \Re(M^{(p_k,z_k)_k}_n(f)) \right|  > R\right)
 \,\le\, \frac{1}{2}\left(4 \cdot \frac18 \cdot \frac78\right)^{N/2} 
 \le\, 2^{-N/2-1}.
\]
An analogous estimate holds for the imaginary part.
If a complex number has an absolute value larger than $\sqrt{2}R$,
then its real or its imaginary part must be larger than $R$.
This gives
\[
 \P\left( \left| I_d(f) - M^{(p_k,z_k)_k}_n(f) \right|  > \sqrt{2}R\right)
 \,\le\, 2^{-N/2}.
\]
By our choice of $N= 2\lceil h(n) \log_2 n \rceil+1$
and since $h(n) \ge \lambda+1/2$,
we obtain that
\[
 \P\left( \left| I_d(f) - M^{(p_k,z_k)_k}_n(f) \right|  \,>\, \sqrt{2}R\right)
 \,\le\, n^{-h(n)} \,\le\, n^{-\lambda - 1/2}.
\]

Now, we use that for any integer $p$ and for any generating vector $z\in \{1:p-1\}^d$, it holds that
\begin{align*}
    | I_d(f) - &Q_p^{z}(f)|  = \bigg| \sum_{\substack{h\in \Z^d\setminus \{0\}\\ h\cdot z\equiv_p 0}}\hat{f}(h) \bigg| \leq \sum_{\substack{h\in \Z^d\setminus \{0\}\\ h\cdot z\equiv_p 0}} |\hat{f}(h)| \\
    & \leq \sqrt{\sum_{h\in \Z^d\setminus \{0\}}|\hat{f}(h)|^2(r_{d,\alpha,\gamma}(h))^2}\cdot \sqrt{\sum_{h\in \Z^d\setminus \{0\}}\frac{1}{(r_{d,\alpha,\gamma}(h))^2}}\\
    & =:  \Vert f \Vert_{d,\alpha,\gamma}\cdot C_{d,\alpha,\gamma}',
\end{align*}
where the first equality follows from the character property of lattice points, see \cite[Proposition~1.12]{DKP22} and the first and second inequalities follow from the triangle inequality and the Cauchy--Schwarz inequality, respectively.
Here, for the constant $C_{d,\alpha,\gamma}'$, since $\lambda>1/2$, we have 
\begin{align*}
    C_{d,\alpha,\gamma}' & = \sqrt{\sum_{h\in \Z^d\setminus \{0\}}\frac{1}{(r_{d,\alpha,\gamma}(h))^2}}=\sqrt{\sum_{\emptyset \neq u\subseteq \{1:d\}}\gamma_u^2 (2\zeta(2\alpha))^{|u|}}\\
    & \leq \left( \sum_{\emptyset \neq u\subseteq \{1:d\}}\gamma_u^{1/\lambda} (2\zeta(2\alpha))^{|u|/(2\lambda)}\right)^{\lambda} \\
    & \leq \left( \sum_{\emptyset \neq u\subseteq \{1:d\}}\gamma_u^{1/\lambda}(2\zeta(\alpha/\lambda))^{|u|}\right)^{\lambda} =  V_d\left(\alpha/\lambda,\gamma^{1/\lambda}\right)^{\lambda}.
\end{align*}
Putting these things together, we obtain
\begin{align*}
\E\, | I_d&(f) - M^{(p_k,z_k)_k}_n(f) |\\
\,&\le\, \P\left( \left| I_d(f) - M^{(p_k,z_k)_k}_n(f) \right|  \,>\, \sqrt{2}R\right) \cdot C_{d,\alpha,\gamma}' \, \Vert f\Vert_{d,\alpha,\gamma} \,+\, \sqrt{2}R\\ 
&\le\, n^{-\lambda-1/2+\delta} \cdot V_d\left(\alpha/\lambda,\gamma^{1/\lambda}\right)^{\lambda} \cdot \Vert f\Vert_{d,\alpha,\gamma} \cdot \left( 1+16\,\sqrt{2}\cdot \tilde{C}_{\lambda,\delta}\right).
\end{align*}
This completes the proof. 
\end{proof}

\begin{proof}[Proof of Theorem~\ref{thm:main}, part~(i)]
    Part (i) in Theorem~\ref{thm:main} follows from Theorem~\ref{thm:randomized-detailed} 
    by choosing $\lambda= \alpha-\varepsilon/2$ and $\delta= \varepsilon/2$.
    Note that, without loss of generality, we may assume that $\varepsilon$ is small enough such that the required restrictions on $\lambda$ and $\delta$ are fulfilled, i.e., $\varepsilon < \min\{\alpha-1/2,2\}$. Since $h$ is assumed to increase slowly, the bound proven in Theorem~\ref{thm:randomized-detailed} applies for all $n\geq n_0$, where
    \[ n_0=\max\left( \min\left\{ n : h(n)\geq \alpha-\varepsilon/2+1/2\right\}, 64V_d(\alpha,\gamma,\alpha-\varepsilon/2)\right). \]
    Moreover, in case of product weights for $\gamma\in (0,1]^{\N}$, the expression $V_d\left(\alpha/\lambda,\gamma^{1/\lambda}\right)$ is bounded above by $M_d(\alpha,\gamma,\lambda)$ as shown in \eqref{eq:rem1}, and $M_d(\alpha,\gamma,\lambda)$ is further bounded independently of the dimension $d$ if $\gamma \in \ell_{1/\alpha}$ holds.
\end{proof}

\section{Analysis of the deterministic error}

The result in the deterministic setting is proven as in \cite{GL22}.

\begin{theorem}\label{thm:deterministic-detailed}
    For any $\alpha>1/2$, 
    any $\gamma=(\gamma_u)_{u\subset \N}$ with $0<\gamma_u\leq 1$, any $n\geq 2$,
    and any $\tau\in(1/2,1)$,
    Algorithm~\ref{alg:median} satisfies
    \[
    e^{\rm det}_{d,\alpha,\gamma}(M_n^{(p_k,z_k)_k}) \,\le\, \inf_{1/2\leq \lambda<\alpha}\left(\frac{4 \sqrt{2}}{(1-\tau)n}V_d\left(\alpha/\lambda,\gamma^{1/\lambda}\right)\right)^{\lambda}
    \]
    with probability at least $1-\frac12(4\tau(1-\tau))^{N/2}$ 
    and $V_d$ as defined in \eqref{eq:def-V}. 
\end{theorem}

\begin{proof}
    Let $a_1,\ldots,a_N$ be real numbers. We can check that 
    \[ \left|\median\{a_1,\ldots,a_N\}\right| \leq \median\{|a_1|,\ldots,|a_N|\}.\]
    For complex numbers $a_1,\ldots,a_N$, this implies
    \[ \left|\median\{a_1,\ldots,a_N\}\right| \leq \sqrt{2} \cdot \median\{|a_1|,\ldots,|a_N|\}.\]
    Fix $(p_1,z_1),\ldots,(p_N,z_N)$. Then, it follows from the above inequality on the median 
    that the deterministic algorithm $M_n^{(p_k,z_k)_{k}}$ satisfies
    \begin{align*}
        e^{\rm det}_{d,\alpha,\gamma}&(M_n^{(p_k,z_k)_k})
        = \sup_{\substack{f\in \mathcal{H}_{d,\alpha,\gamma}\\ \| f\|_{d,\alpha,\gamma}\leq 1}}\left| I_d(f) - M_n^{(p_k,z_k)_k}(f) \right| \\
        &= \sup_{\substack{f\in \mathcal{H}_{d,\alpha,\gamma}\\ \| f\|_{d,\alpha,\gamma}\leq 1}}\left| \median\left\{I_d(f) - Q_{p_1}^{z_1}(f),
        \hdots, I_d(f) - Q_{p_N}^{z_N}(f)\right\} \right|\\
        &\le \sup_{\substack{f\in \mathcal{H}_{d,\alpha,\gamma}\\ \| f\|_{d,\alpha,\gamma}\leq 1}} \sqrt{2}\cdot \median\left\{\left|I_d(f) - Q_{p_1}^{z_1}(f)\right|,
        \hdots, \left|I_d(f) - Q_{p_N}^{z_N}(f)\right|\right\} \\
        &\le \sqrt{2} \cdot \median\Big\{ e^{\rm det}_{d,\alpha,\gamma}(Q_{p_k}^{z_k}) \ : \ 1\le k \le N \Big\},
    \end{align*}
    compare the proof of \cite[Proposition~3.2]{GSM24}.

    Let $p$ be uniformly distributed on $\mathcal{P}_n$ and let
    $z$ be uniformly distributed on $\{1:p-1\}^d$. Lemma~\ref{lem:set_good_vectors} proves that 
    \[ e^{\rm det}_{d,\alpha,\gamma}(Q_p^z)\leq \inf_{1/2\leq \lambda<\alpha}\left(\frac{4}{(1-\tau)n}V_d\left(\alpha/\lambda,\gamma^{1/\lambda}\right)\right)^{\lambda} \]
    holds with probability at least $\tau$ for any fixed $\tau\in (0,1)$. Thus, if we choose a random prime $p_k$ uniformly from $\mathcal{P}_n$ and a random vector $z_k$ uniformly from $\{1:p_k-1\}^d$ for all $1\leq j\leq k$, we have
    \begin{align*}
        & \P\left( e^{\rm det}_{d,\alpha,\gamma}(M_n^{(p_k,z_k)_k}) > \sqrt{2} \inf_{1/2\leq \lambda<\alpha}\left(\frac{4}{(1-\tau)n}V_d\left(\alpha/\lambda,\gamma^{1/\lambda}\right)\right)^{\lambda}\right) \\
        & \quad \leq\, \P\left( \median\left\{ e^{\rm det}_{d,\alpha,\gamma}(Q_{p_k}^{z_k})\right\} > \inf_{1/2\leq \lambda<\alpha}\left(\frac{4}{(1-\tau)n}V_d\left(\alpha/\lambda,\gamma^{1/\lambda}\right)\right)^{\lambda}\right)\\
        & \quad \le\, \frac{1}{2}(4\tau(1-\tau))^{N/2},
    \end{align*}
    for any $\tau\in (1/2, 1)$, where we used Lemma~\ref{lem:median}.
\end{proof}

\begin{proof}[Proof of Theorem~\ref{thm:main}, part~(ii)]
    Part (ii) in Theorem~\ref{thm:main} follows from Theorem~\ref{thm:deterministic-detailed} 
    by choosing $\lambda= \alpha-\varepsilon$ and $\tau=7/8$. With this choice and 
    $N= 2\lceil h(n) \log_2 n \rceil +1$, we have
    \[ \frac12(4\tau(1-\tau))^{N/2}\leq 2^{-N/2-1}\leq n^{-h(n)}.\]
    Again, in case of product weights for $\gamma\in (0,1]^{\N}$, $V_d\left(\alpha/\lambda,\gamma^{1/\lambda}\right)$ is bounded above by $M_d(\alpha,\gamma,\lambda)$ as shown in \eqref{eq:rem1}, and $M_d(\alpha,\gamma,\lambda)$ is further bounded independently of the dimension $d$ if $\gamma \in \ell_{1/\alpha}$.
\end{proof}

\section{Numerical examples}

We conclude this paper with numerical experiments to validate our theoretical results.
We start with two test functions that have been used in previous work:
\begin{align*}
f_1(x) & =\prod_{j=1}^{d}\left(1+\frac{|4x_j-2|-1}{j^{c_1}}\right),\\
f_2(x) & = \prod_{j=1}^{d}\left(1+\frac{(x_j-1/2)^2\, \sin(2\pi x_j-\pi)}{j^{c_2}}\right).
\end{align*}
The first test function was used in \cite{CXZ24,DGS22} and the second in \cite{G24}.
One can check that, for $d=1$, $\hat{f}_1(h)=O(|h|^{-2})$ and $\hat{f}_2(h)=O(|h|^{-3})$.
This implies that $f_1\in \mathcal{H}_{d,3/2-\epsilon,\gamma}$ and $f_2\in \mathcal{H}_{d,5/2-\epsilon,\gamma}$, respectively, for arbitrarily small $\epsilon>0$.

In what follows, we use Algorithm~\ref{alg:median} to integrate $f_1$ and $f_2$ with various values of $c_1$ and $c_2$. 
We estimate the expected error by the sample average over $100$ independent realizations:
\begin{align}\label{eq:error_app}
    \E_{\omega}\left| I_d(f) - M_n^{\omega}(f) \right| \approx \frac{1}{100}\sum_{r=1}^{100}\left| I_d(f) - M_n^{\omega^{(r)}}(f) \right|,
\end{align}
where we have $I_d(f_i)=1$ for $i\in \{1,2\}$.
In all cases, we set $h(n) = \max(1, \log \log n)$.

The results for dimension $d=20$ are shown in Figure~\ref{fig:result}.
The figures display the error decay as functions of $n$ for various values of the parameter $c_i$.
As a reference, the lines corresponding to $n^{-1}$, $n^{-1.5}$, and $n^{-2}$ are plotted for $f_1$, while the lines corresponding to $n^{-2}$, $n^{-2.5}$, and $n^{-3}$ are plotted for $f_2$.

From the theoretical results, an error decay of order $n^{-2+\epsilon}$ is expected as the optimal rate for $f_1$, while an error decay of order $n^{-3+\epsilon}$ is expected for $f_2$.
To ensure tractable error bounds with the desired rates, we require the conditions $c_1> 2$ and $c_2 > 3$. 
In fact, 
larger values of $c_1$ and $c_2$ lead to a better rate of convergence, approaching the theoretically optimal rate. 
Linear regression for the range $n \geq 10^2$ gives empirical rates of order $n^{-1.974}$ and $n^{-2.683}$
for the largest values of $c_i$, respectively.

\begin{figure}[t]
    \centering
    \begin{subfigure}{0.48\textwidth}
        \centering
        \includegraphics[width=\textwidth]{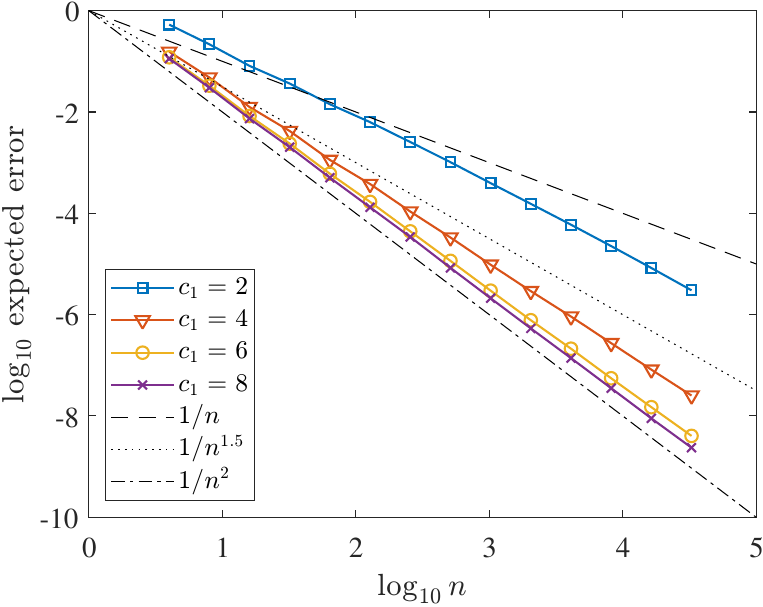}
        \caption{$f_1$}
    \end{subfigure}
    \begin{subfigure}{0.48\textwidth}
        \centering
        \includegraphics[width=\textwidth]{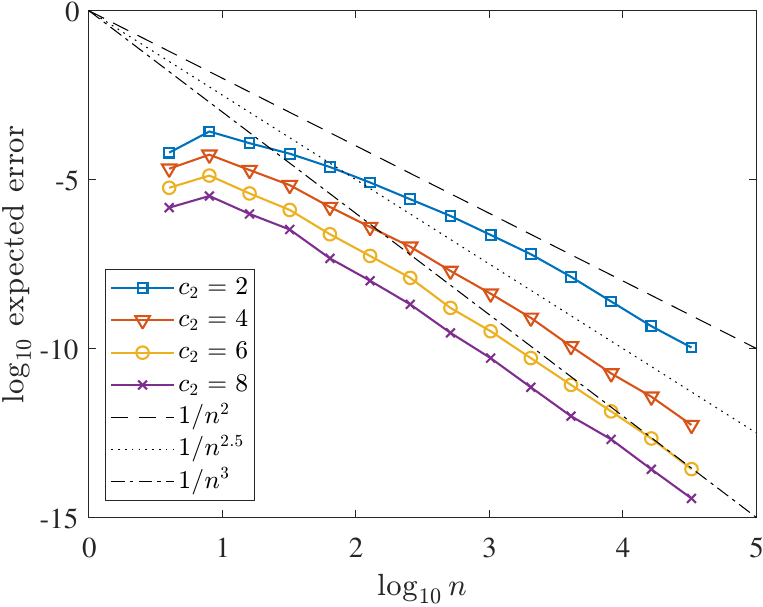}
        \caption{$f_2$}
    \end{subfigure}
    \caption{Results for the test functions $f_1$ and $f_2$ with $d=20$.}
    \label{fig:result}
\end{figure}

We like to recall the main benefit of our algorithm compared to other methods, which is its universality. We did not need to choose any parameters of the algorithm depending on the properties of the test functions. 
To emphasize this point further, let us also test the following parameterized family of functions:
\[
    f_{a,c}(x) = \prod_{j=1}^{d}\left(1+\frac{g_a(x_j)-\int_0^1 g_a(t) \,{\rm d}t}{j^{c}}\right),
\]
for $a>0$ and $c>a+1$, where
\[
 g_a(x) = \Big|x-\frac12\Big|^a \cdot \exp\left( \frac{1}{(2x-1)^2-1} \right).
\]
This function belongs to the Korobov space with smoothness arbitrarily close to $a+1/2$.
Thus, we hope to observe a randomized error rate close to $a+1$.

By choosing product weights $\gamma_j=j^{-\beta}$ with $\beta\in (a+1/2,c-1/2)$
and putting $\alpha = a + 1/2 - \delta$ with $\delta \in (0,a)$, 
one can check that $\gamma \in \ell_{1/\alpha}$ and that the norm of $f_{a,c}$ in $\mathcal{H}_{d,\alpha,\gamma}$ is bounded above independently of $d$, ensuring that Theorem~\ref{thm:main} applies with constants independent of $d$.
Namely, for any $\varepsilon>0$, there is a constant $C>0$ such that
\[
 \E\,|M_n(f_{a,c}) - I_d(f_{a,c})|
 \,\le\, C \,n^{-a-1+\varepsilon}.
\]
The constant $C$ is independent of the dimension $d$ and only depends on the parameters $a, c$ and $\varepsilon$.

In what follows, we test various values of $a\in \{0.1, 0.5, 1, 2.2, 3.4, 3.9\}$. We fix the dimension $d=50$ and the parameter $c=2a+1$. 
Note that to compute the true integral of $g_a$, we use the Matlab function \texttt{integral} with the absolute error tolerance of $10^{-14}$.
The expected error is estimated by \eqref{eq:error_app}, in which $I_d(f_{a,c})=1$.

\begin{figure}[t]
    \centering
    \begin{subfigure}{0.48\textwidth}
        \centering
        \includegraphics[width=\textwidth]{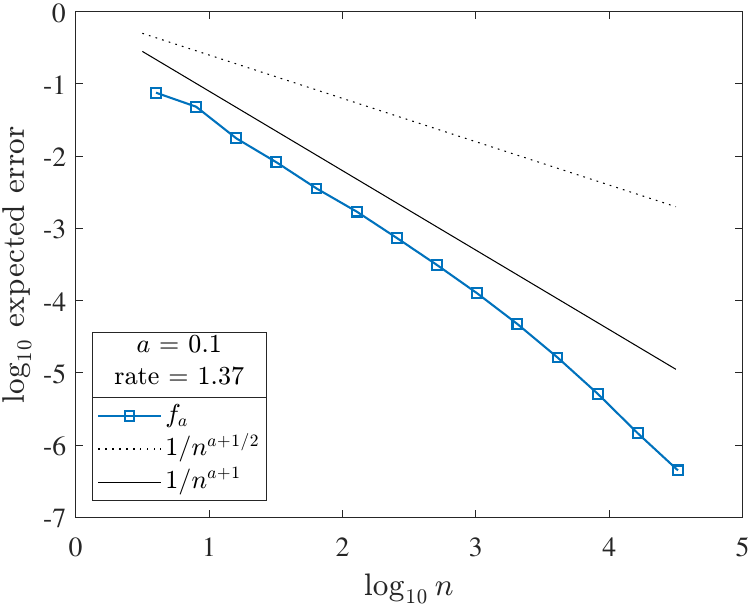}
        \caption{$a=0.1$}
    \end{subfigure}
    \begin{subfigure}{0.48\textwidth}
        \centering
        \includegraphics[width=\textwidth]{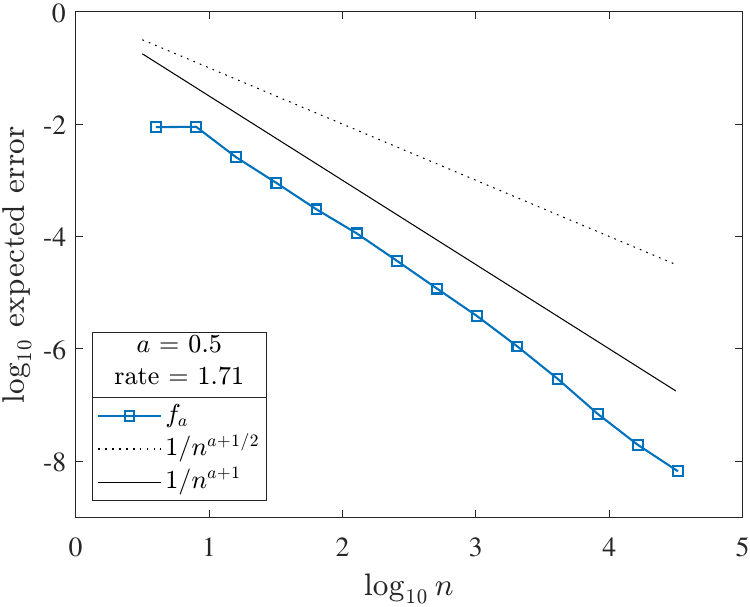}
        \caption{$a=0.5$}
    \end{subfigure}\\
    \vspace{10pt}
    \begin{subfigure}{0.48\textwidth}
        \centering
        \includegraphics[width=\textwidth]{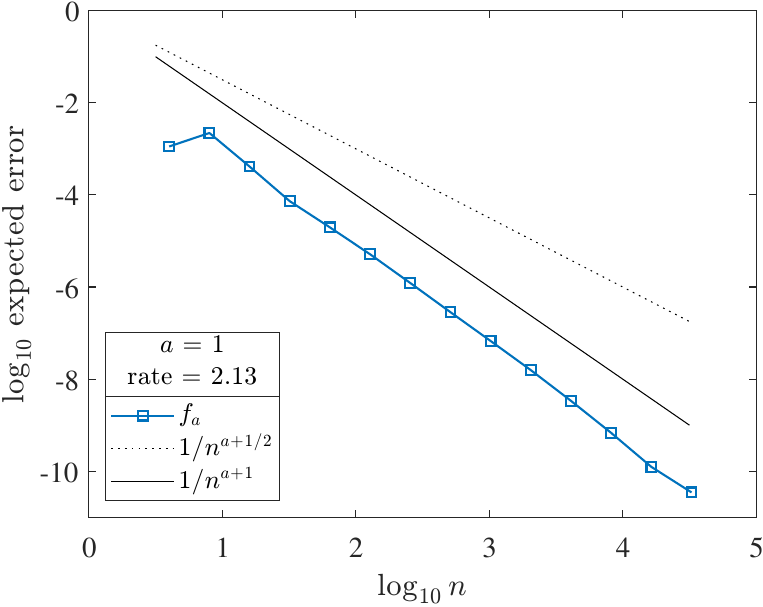}
        \caption{$a=1$}
    \end{subfigure}
    \begin{subfigure}{0.48\textwidth}
        \centering
        \includegraphics[width=\textwidth]{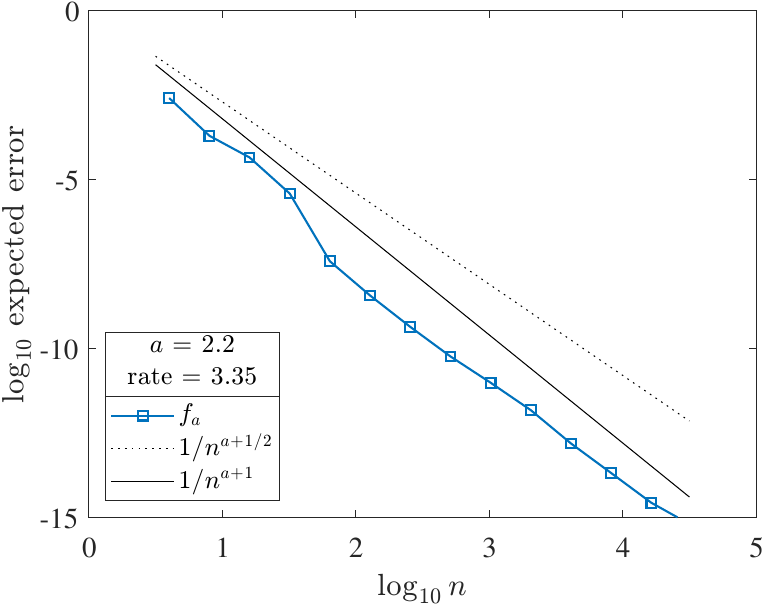}
        \caption{$a=2.2$}
    \end{subfigure}\\
    \vspace{10pt}
    \begin{subfigure}{0.48\textwidth}
        \centering
        \includegraphics[width=\textwidth]{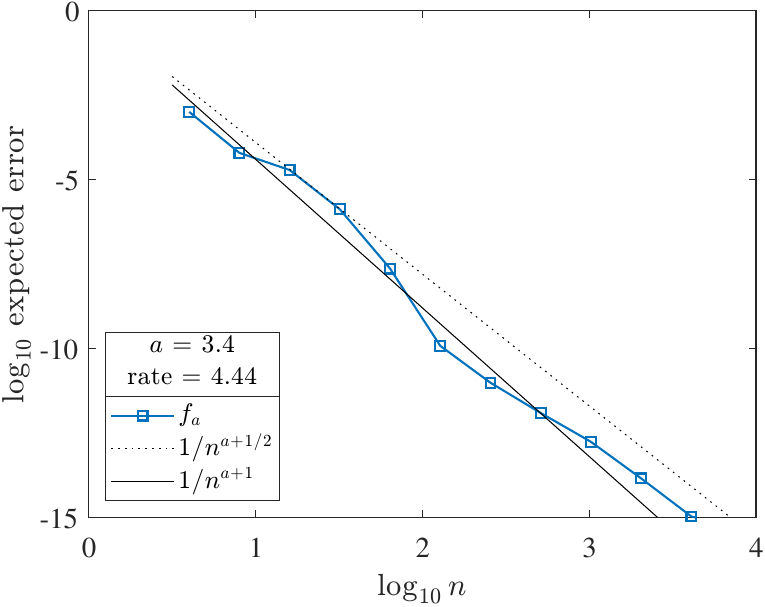}
        \caption{$a=3.4$}
    \end{subfigure}
    \begin{subfigure}{0.48\textwidth}
        \centering
        \includegraphics[width=\textwidth]{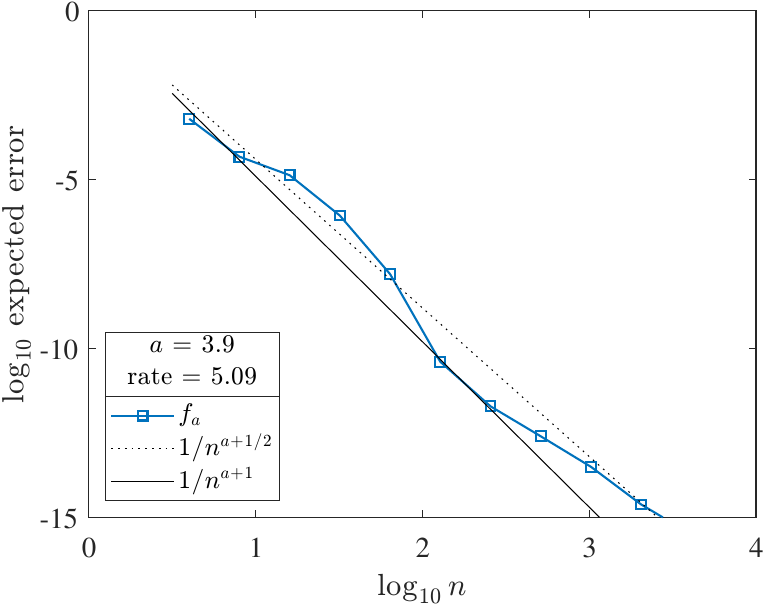}
        \caption{$a=3.9$}
    \end{subfigure}
    \caption{Results for the test function $f_{a,c}$ with $d=50$ and $c=2a+1$. Each subfigure corresponds to a different value of $a$.}
    \label{fig:result2}
\end{figure}

The results are shown in Figure~\ref{fig:result2}, where each subfigure corresponds to a different value of $a$. 
To estimate the empirical error rates, linear regression is applied over the range $n \geq 10$ up to the largest $n$ where the error remains above $10^{-13}$.
The estimated empirical rates are indicated in each subfigure. 
As observed, for all the cases we considered, our algorithm achieves the theoretically optimal (or even better) error rates, automatically exploiting the smoothness of the integrands. This observation supports our theoretical results.

Lastly, we also consider a non-periodic test function:
\begin{align*}
 f(x) &  = \prod_{j=1}^{d}\Big(1+\frac{\theta^j}{8}\left( 31-84x_j^2+8x_j^3+70x_j^4-28x_j^6+8x_j^7\right. \\
       & \qquad \qquad \left.-16\cos(1)-16\sin(x_j)\right)\Big).
\end{align*}
This function was considered also in \cite{DNP14}.
We apply the tent transformation to Algorithm~\ref{alg:median}, as explained in Remark~\ref{rem:nonperiodic}, 
and consider various values of $\theta$.
Considering the half-period cosine expansions of $f$, the cosine coefficients decay as $O(h^{-2})$.
This means that $f$ belongs to the half-period cosine space with smoothness arbitrarily close to $3/2$. We require $|\theta|< 1$ in order to obtain tractable error bounds. 
In the numerical experiments, we hope to see a rate close to $2$.

The numerical results in dimension $d=10$ are shown in Figure~\ref{fig:result-nonper}. One can observe that
varying $\theta$ significantly affects the magnitude of the error. 
With $\theta = 0.9$, the empirical rate is of order $n^{-1.020}$, whereas for $\theta = 0.1$, the rate improves to $n^{-1.906}$, which is very close to the optimal rate.


\begin{figure}[t]
    \centering
    \includegraphics[width=0.48\textwidth]{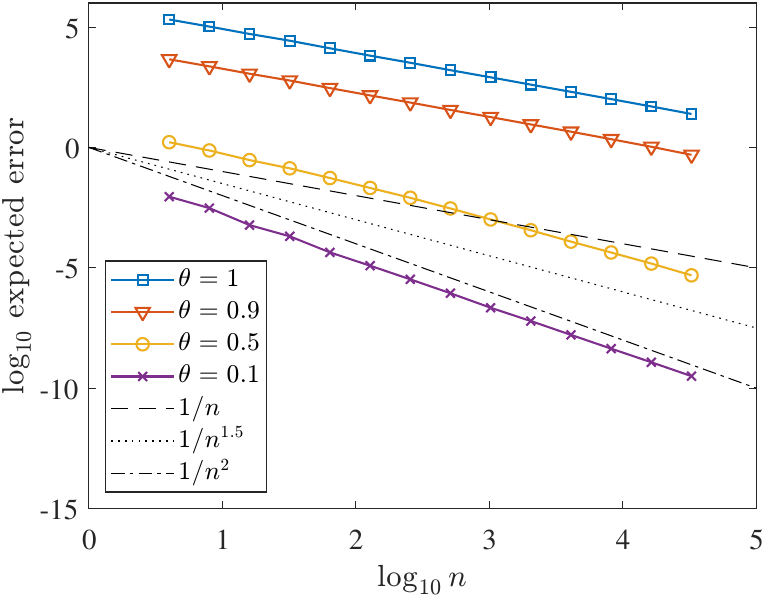}
    \caption{Results for the nonperiodic test function with $d=10$.}
    \label{fig:result-nonper}
\end{figure}

\section*{Acknowledgments}
This research was initiated at the 16th International Conference on Monte Carlo and Quasi-Monte Carlo Methods in Scientific Computing (MCQMC 2024). The authors are grateful to the conference organizers for their kind invitation.

\bibliographystyle{amsplain}
\bibliography{ref}

\providecommand{\bysame}{\leavevmode\hbox to3em{\hrulefill}\thinspace}
\providecommand{\MR}{\relax\ifhmode\unskip\space\fi MR }
\providecommand{\MRhref}[2]{%
  \href{http://www.ams.org/mathscinet-getitem?mr=#1}{#2}
}
\providecommand{\href}[2]{#2}
\begin{thebibliography}{10}

\bibitem{Bah61}
N.~S. Bahvalov, \emph{Estimates in the mean of the remainder term of quadratic
  formulas}, \v Z. Vy\v cisl. Mat i Mat. Fiz. \textbf{1} (1961), 64--77
  (Russian).

\bibitem{CXZ24}
L.~Chen, M.~Xu, and H.~Zhang, \emph{A random integration algorithm for
  high-dimensional function spaces}, arXiv preprint arXiv:2406.16627 (2024),
  23~pp.

\bibitem{DG21}
J.~Dick and T.~Goda, \emph{Stability of lattice rules and polynomial lattice
  rules constructed by the component-by-component algorithm}, J. Comput. Appl.
  Math. \textbf{382} (2021), Paper No. 113062, 16~pp.

\bibitem{DGS22}
J.~Dick, T.~Goda, and K.~Suzuki, \emph{Component-by-component construction of
  randomized rank-1 lattice rules achieving almost the optimal randomized error
  rate}, Math. Comp. \textbf{91} (2022), no.~338, 2771--2801.

\bibitem{DKLP15}
J.~Dick, P.~Kritzer, G.~Leobacher, and F.~Pillichshammer, \emph{A reduced fast
  component-by-component construction of lattice points for integration in
  weighted spaces with fast decreasing weights}, J. Comput. Appl. Math.
  \textbf{276} (2015), 1--15.

\bibitem{DKP22}
J.~Dick, P.~Kritzer, and F.~Pillichshammer, \emph{Lattice rules---numerical
  integration, approximation, and discrepancy}, Springer Series in
  Computational Mathematics, vol.~58, Springer, Cham, 2022.

\bibitem{DNP14}
J.~Dick, D.~Nuyens, and F.~Pillichshammer, \emph{Lattice rules for nonperiodic
  smooth integrands}, Numer. Math. \textbf{126} (2014), no.~2, 259--291.

\bibitem{DSWW06}
J.~Dick, I.~H. Sloan, X.~Wang, and H.~Wo\'zniakowski, \emph{Good lattice rules
  in weighted {K}orobov spaces with general weights}, Numer. Math. \textbf{103}
  (2006), no.~1, 63--97.

\bibitem{EKNO21}
A.~Ebert, P.~Kritzer, D.~Nuyens, and O.~Osisiogu, \emph{Digit-by-digit and
  component-by-component constructions of lattice rules for periodic functions
  with unknown smoothness}, J. Complexity \textbf{66} (2021), Paper No. 101555,
  37~pp.

\bibitem{GDMS24}
M.~Gnewuch, J.~Dick, L.~Markhasin, and W.~Sickel, \emph{{Q}{M}{C} integration
  based on arbitrary $(t, m, s)$-nets yields optimal convergence rates on
  several scales of function spaces}, arXiv preprint arXiv:2409.12879 (2024),
  56~pp.

\bibitem{G24}
T.~Goda, \emph{A randomized lattice rule without component-by-component
  construction}, arXiv preprint arXiv:2403.02660 (2024), 21~pp.

\bibitem{GL22}
T.~Goda and P.~L'Ecuyer, \emph{Construction-free median quasi--{M}onte {C}arlo
  rules for function spaces with unspecified smoothness and general weights},
  SIAM J. Sci. Comput. \textbf{44} (2022), no.~4, A2765--A2788.

\bibitem{GSM24}
T.~Goda, K.~Suzuki, and M.~Matsumoto, \emph{A universal median quasi--{M}onte
  {C}arlo integration}, SIAM J. Numer. Anal. \textbf{62} (2024), no.~1,
  533--566.

\bibitem{HR23}
J.~Hofstadler and D.~Rudolf, \emph{Consistency of randomized integration
  methods}, J. Complexity \textbf{76} (2023), Paper No. 101740, 10~pp.

\bibitem{KN17}
D.~Krieg and E.~Novak, \emph{A universal algorithm for multivariate
  integration}, Found. Comput. Math. \textbf{17} (2017), no.~4, 895--916.

\bibitem{KKNU19}
P.~Kritzer, F.~Y. Kuo, D.~Nuyens, and M.~Ullrich, \emph{Lattice rules with
  random {$n$} achieve nearly the optimal {$\mathcal{O}(n^{-\alpha-1/2})$}
  error independently of the dimension}, J. Approx. Theory \textbf{240} (2019),
  96--113.

\bibitem{KNR19}
R.~J. Kunsch, E.~Novak, and D.~Rudolf, \emph{Solvable integration problems and
  optimal sample size selection}, J. Complexity \textbf{53} (2019), 40--67.

\bibitem{KR19}
R.~J. Kunsch and D.~Rudolf, \emph{Optimal confidence for {M}onte {C}arlo
  integration of smooth functions}, Adv. Comput. Math. \textbf{45} (2019),
  no.~5-6, 3095--3122.

\bibitem{Kuo03}
F.~Y. Kuo, \emph{Component-by-component constructions achieve the optimal rate
  of convergence for multivariate integration in weighted {K}orobov and
  {S}obolev spaces}, J. Complexity \textbf{19} (2003), no.~3, 301--320.

\bibitem{KNW23}
F.~Y. Kuo, D.~Nuyens, and L.~Wilkes, \emph{Random-prime-fixed-vector randomised
  lattice-based algorithm for high-dimensional integration}, J. Complexity
  \textbf{79} (2023), Paper No. 101785, 28~pp.

\bibitem{KSW07}
F.~Y. Kuo, I.~H. Sloan, and H.~Wo\'zniakowski, \emph{Periodization strategy may
  fail in high dimensions}, Numer. Algorithms \textbf{46} (2007), no.~4,
  369--391.

\bibitem{NUU17}
V.~K. Nguyen, M.~Ullrich, and T.~Ullrich, \emph{Change of variable in spaces of
  mixed smoothness and numerical integration of multivariate functions on the
  unit cube}, Constr. Approx. \textbf{46} (2017), no.~1, 69--108.

\bibitem{NC06}
D.~Nuyens and R.~Cools, \emph{Fast algorithms for component-by-component
  construction of rank-1 lattice rules in shift-invariant reproducing kernel
  {H}ilbert spaces}, Math. Comp. \textbf{75} (2006), no.~254, 903--920.

\bibitem{Pan24}
Z.~Pan, \emph{Automatic optimal-rate convergence of randomized nets using
  median-of-means}, arXiv preprint arXiv:2411.01397 (2024), 34~pp.

\bibitem{PO23}
Z.~Pan and A.~B. Owen, \emph{Super-polynomial accuracy of one dimensional
  randomized nets using the median of means}, Math. Comp. \textbf{92} (2023),
  no.~340, 805--837.

\bibitem{PO24}
\bysame, \emph{Super-polynomial accuracy of multidimensional randomized nets
  using the median-of-means}, Math. Comp. \textbf{93} (2024), no.~349,
  2265--2289.

\bibitem{SR02}
I.~H. Sloan and A.~V. Reztsov, \emph{Component-by-component construction of
  good lattice rules}, Math. Comp. \textbf{71} (2002), no.~237, 263--273.

\bibitem{Ull17}
M.~Ullrich, \emph{A {M}onte {C}arlo method for integration of multivariate
  smooth functions}, SIAM J. Numer. Anal. \textbf{55} (2017), no.~3,
  1188--1200.

\end{thebibliography}


\end{document}